\documentclass{amsart}
\usepackage{amssymb}
\usepackage{mathtools}
\usepackage[only,llbracket,rrbracket]{stmaryrd}
\usepackage[svgnames]{xcolor}
\usepackage[unicode,
  colorlinks=true,
  linktocpage=true,
  citecolor=OliveDrab,
  linkcolor=DarkMagenta,
  urlcolor=DarkMagenta,
  pdftitle={Coskeletality and the higher Segal conditions},
  pdfauthor={Philip Hackney}
  ]{hyperref}
\usepackage[capitalise,noabbrev]{cleveref}
\usepackage{tikz-cd}
\usepackage{microtype}
\usepackage[T1]{fontenc}
\usepackage{enumitem}

\setlist[enumerate,1]{left=0pt, label={\arabic*.}, ref=\arabic*}

\usepackage{comment}

\newcommand{\cube}[1]{\llbracket #1 \rrbracket}

\newcommand{\lowev}{\mathsf{le}}
\newcommand{\lowod}{\mathsf{lo}}
\newcommand{\uppev}{\mathsf{ue}}
\newcommand{\uppod}{\mathsf{uo}}

\newcommand{\ps}{\mathcal{P}}
\newcommand{\op}{\textup{op}}

\newcommand{\sset}{\mathsf{sSet}}

\newcommand{\bim}{{\bar \imath}}
\newcommand{\bjm}{{\bar \jmath}}

\newcommand{\decbot}{\operatorname{dec}_\bot}
\newcommand{\dectop}{\operatorname{dec}_\top}

\DeclareMathOperator{\cosk}{cosk}
\DeclareMathOperator{\sk}{sk}
\DeclareMathOperator{\map}{map}

\newtheorem{theorem}{Theorem}
\newtheorem{proposition}[theorem]{Proposition}

\newtheorem{lemma}[theorem]{Lemma}

\theoremstyle{definition}
\newtheorem{definition}[theorem]{Definition}

\theoremstyle{remark}
\newtheorem{remark}[theorem]{Remark}

\begin{document}

\title{Coskeletality and the higher Segal conditions}

\author{Philip Hackney}
\address{Department of Mathematics, University of Louisiana at Lafayette}
\email{philip@phck.net} 
\urladdr{http://phck.net}

\thanks{
This work was supported by a grant from the Simons Foundation (\#850849, PH). 
The author was partially supported by the Louisiana Board of Regents through the Board of Regents Support fund LEQSF(2024-27)-RD-A-31.
}

\date{\today}

\begin{abstract}
A simplicial set is $d$-Segal if and only if it is $(d{+}1)$-coskeletal and satisfies the $d$-Segal condition in the two lowest relevant simplicial dimensions.
\end{abstract}

\maketitle

A simplicial set $X$ is $n$-coskeletal if it is determined by its simplices of dimension at most $n$, together with unique fillers for higher-dimensional simplex boundaries.\footnote{Precisely, $X$ is right Kan extended along $\Delta_{\leq n} \hookrightarrow \Delta$ from its restriction to $\Delta_{\leq n}$.} 
It is well-known that a simplicial set is the nerve of a category if and only if it is 2-coskeletal\footnote{In the reverse direction, one may actually assume that the simplicial set is 3-coskeletal rather than 2-coskeletal.} and has unique fillers for inner horns in simplicial dimension 2 and 3.\footnote{\cite{nlab:simplicial_skeleton}, \cite[Lemma 5.2]{Riehl:SSCAQC}}
There is a similar result for 2-Segal sets \cite{DyckerhoffKapranov:HSS} (also known as discrete decomposition spaces \cite{GKT1}):
a simplicial set is 2-Segal if and only if it is 3-coskeletal and satisfies the 2-Segal condition for the square and the pentagon (that is, it satisfies the 2-Segal conditions in simplicial dimensions 3 and 4).\footnote{\cite[Corollary 1.7]{BOORS:2SSetsWC}, \cite[\S4.1]{Stern:P2SC} -- the latter source observes that one may replace 3-coskeletal by 4-coskeletal in the reverse direction.}
In this paper, we prove an analogous result for $d$-Segal sets, resolving a question of Walker Stern: a simplicial set is upper (resp.\ lower) $d$-Segal if and only if it is $(d{+}1)$-coskeletal and satisfies the upper (resp.\ lower) $d$-Segal conditions in simplicial dimensions $d+1$ and $d+2$ (the lowest two non-vacuous dimensions).
See \cref{thm main theorem} for a slightly stronger statement.
The corresponding statement for simplicial spaces is false (\cref{rmk simplicial anima}).

The higher Segal conditions were introduced by Dyckerhoff and Kapranov \cite{Dyckerhoff:CPOC,DyckerhoffKapranov:HSS} as exactness conditions arising from triangulations of cyclic polytopes, generalizing the classical Segal conditions underlying models for $(\infty,1)$-categories.
The case $d=2$ has proven especially fruitful, with applications to Hall algebras, incidence coalgebras, and other constructions in representation theory, algebraic combinatorics, and algebraic K-theory \cite{Dyckerhoff:HCAHA,GKT:DSC}. 
The situation for general $d$ is considerably less developed: besides \cite{Dyckerhoff:CPOC,Poguntke:HSSAKT}, see \cite{DyckerhoffJassoWalde:SSHART,HackneyLynd:HSSPG,Walde:HSSHE} for some of the first steps in this direction.

Our criterion allows for an effective, finite check of $d$-Segality for simplicial sets with finitely many nondegenerate simplices (for $d$ suitably large).
This rests on \cite[Theorem 3.19]{KRRZ:LTSS}, which states that $n$-skeletal simplicial sets (with $n > 1$) are automatically $(2n{-}1)$-coskeletal.
This was used to compute that $\Delta^n/\partial \Delta^n$ is $2n$-Segal but not lower $(2n{-}1)$-Segal for $2 \leq n \leq 7$; we proceeded to prove this for general $n \geq 1$ \cite{Hackney:SdSS}. 
Thus the strictness of the hierarchy of higher Segal conditions is visible in rather simple and familiar examples.

\section{Simplicial conventions}
In this paper $\Delta$ will refer to the category whose objects are nonempty finite sets of integers and whose morphisms are order-preserving maps.
The usual simplicial indexing category is the skeletal subcategory of $\Delta$ whose objects are the intervals $[n] = [0,n] = \{0, 1, \dots, n\}$ for $n\geq 0$.
We'll write $X(S)$ for $X$ evaluated at $S\in \Delta$; we occasionally abbreviate $X([n])$ to $X_n$.

The opposite $X^\op$ of a simplicial object $X$ is given by $X^\op(S) = X(-S)$, and sends $f\colon S \to T$ to
\[
	X(f^-) \colon X(-T) \to X(-S),
\]
where $f^- \colon (-S) \to (-T)$ is given by $x\mapsto -f(-x)$ for $x\in -S$.

If $S\in \Delta$ is a finite ordered set with more than one element and $i\in S$, we'll write $S_i = S\setminus \{i\}$. 
If $X$ is a simplicial set, we'll write $e_i \colon X(S) \to X(S_i)$ for the restriction.
This is essentially the face map $d_i \colon X_n \to X_{n-1}$ (where $n+1 = |S|$), but has the more convenient interchange property $e_i e_j = e_j e_i$ for $i\neq j$ in $S$.

\subsection{D\'ecalage} 
The \emph{upper d\'ecalage} of $X$, denoted $\dectop X$, is obtained by shifting simplicial dimension by one (i.e.\ $(\dectop X)_n = X_{n+1}$ for $n\geq 0$) and forgetting about the top face and degeneracy maps \cite[VI.1]{Illusie:CCD2}.
More concretely in our setting: $(\dectop X)(S)$ is defined to be $X(S^\triangleright)$ where $S^\triangleright = S \cup \{ \max(S) + 1 \}$, and, if $f\colon S \to T$ is in $\Delta$, then $(\dectop X) (f) \colon (\dectop X)(T) \to (\dectop X)(S)$ is defined to be $X(f^\triangleright)$ where $f^\triangleright \colon S^\triangleright \to T^\triangleright$ is the extension of $f$ which preserves maximal elements.
Similarly the \emph{lower d\'ecalage} $\decbot X$ is given by $(\decbot X)(S) = X(S^\triangleleft)$ and $(\decbot X)(f) = X(f^\triangleleft)$, where $S^\triangleleft = \{\min(S) - 1\} \cup S$, and $f^\triangleleft$ is the extension which  preserves minimal elements.
Since $(f^\triangleright)^- = (f^-)^\triangleleft$, 
the simplicial objects $\dectop (X^\op)$ and $(\decbot X)^\op$ are equal. 

\subsection{Coskeletality}

If $S \in \Delta$ and a simplicial set $X$ are fixed in a discussion, a symbol like $f_i$ will generally be used for an element of $X(S_i)$.

\begin{definition}
Let $X$ be a simplicial set, $U \subset S$, and $f_i \in X(S_i)$ a collection of elements (for $i\in U$).
\begin{itemize}
\item If $e_i f_j = e_j f_i \in X(S_{i,j})$ for all $i\neq j \in U$, we say $\{f_i\}_{i\in U}$ is \emph{compatible}.
\item Suppose additionally that $t \in S \setminus U$ and $f_t \in X(S_t)$. If $e_i f_t = e_t f_i$ for all $i\in U$, then we say $f_t$ is \emph{compatible with} the collection $\{f_i\}_{i\in U}$.
\end{itemize}
\end{definition}

We'll use the following simple compatibility lemma repeatedly.

\begin{lemma}\label{triv lemma}
Let $f_i \in X(S_i)$ for each $i$ in a three element subset $\{p,q,t\} \subset S$. 
If $e_qf_t = e_tf_q$ and $e_qf_p = e_p f_q$, then $e_q e_p f_t = e_q e_t f_p$.
\end{lemma}
\begin{proof} $e_q e_p f_t = e_p e_q f_t = e_p e_t f_q = e_t e_p f_q = e_t e_q f_p = e_q e_t f_p$.
\end{proof}

A compatible collection where $U=S=[m]$ is called an \emph{$m$-sphere in $X$} in \cite[Definition 3.2]{KRRZ:LTSS}.
Coskeletality is about unique filling of $m$-spheres: 

\begin{definition}\label{def n-coskel}
A simplicial set $X\in \sset$ is \emph{$n$-coskeletal} if, whenever $S = [m]$ for $m > n$ and $\{ f_i \}_{i\in S}$ is a compatible collection, there exists a unique $F \in X_m$ such that $e_i F = f_i$ for all $0\leq i \leq m$.
\end{definition}

There are multiple equivalent formulations of $n$-coskeletality. 
For instance we could ask that $X$ is orthogonal to the boundary inclusions $\partial \Delta^m \to \Delta^m$ for all $m > n$, or we could ask that it is in the image of the right Kan extension from $n$-truncated simplicial sets (presheaves on $\Delta_{\leq n}$).
See also \cite[\href{https://kerodon.net/tag/051Z}{Tag 051Z}]{kerodon}.

\section{Background: Higher Segal Conditions}

We closely follow the presentation of the higher Segal conditions from \cite{HackneyLynd:HSSPG}, which is based on Walde's theorem \cite{Walde:HSSHE}.
See \cite{Dyckerhoff:CPOC} for a recent survey on higher Segal spaces.

A \emph{gapped subset} $I$ of a finite totally ordered set $S$ is a proper subset which does not contain a pair of adjacent elements of $S$.
We will often abbreviate `gapped set of cardinality $k+1$' by `gapped set' when the number $k$ is fixed in the discussion.

Given a gapped set $I\subset S$ of cardinality $k+1$, there is a cube
\[
	\cube{I} = \cube{I\subset S} \colon \ps(I) \to \Delta^\op
\]
sending $U \subset I$ to $S\setminus U$.
If $X \colon \Delta^\op \to \mathcal{C}$ is a simplicial object in a category or $\infty$-category $\mathcal{C}$, we write $X\cube{I} \colon \ps(I) \to \mathcal{C}$ for the corresponding composite cube.

\begin{definition}[Higher Segal conditions]\label{def higher Segal}
Let $k$ be a nonnegative integer and $X$ a simplicial object in $\mathcal{C}$.
Consider the collection of gapped subsets $I \subset [n]$ of cardinality $k+1$ as $n$ varies and the associated collection of cubes $X\cube{I} \colon \ps(I) \to \mathcal{C}$.
We say that $X$ is
\begin{enumerate}
\item \emph{lower $(2k{-}1)$-Segal} if $X\cube{I}$ is cartesian for all such $I$,
\item \emph{lower $2k$-Segal} if $X\cube{I}$ is cartesian whenever $0 \notin I$,
\item \emph{upper $2k$-Segal} if $X\cube{I}$ is cartesian whenever $n \notin I$, and
\item \emph{upper $(2k{+}1)$-Segal} if $X\cube{I}$ is cartesian whenever $0 \notin I$ and $n \notin I$.
\end{enumerate}
\end{definition}

We say $X$ is $d$-Segal if it is both upper and lower $d$-Segal.
The lower 1-Segal condition is equivalent to the usual Segal condition.
When $k=0$, all four cases coincide, and amount to $X$ being a constant simplicial object (\cite[Example~3.9]{Dyckerhoff:CPOC}, \cite[Remark~3.11]{HackneyLynd:HSSPG}).
Poguntke showed that if $X$ is upper or lower $d$-Segal, then it is also both upper and lower $(d{+}1)$-Segal (\cite[Theorem~5.1]{Dyckerhoff:CPOC}, \cite[Proposition~3.14]{HackneyLynd:HSSPG}, \cite[Proposition~2.10]{Poguntke:HSSAKT}).

\begin{remark}
As mentioned in \cite[Remark 3.8]{HackneyLynd:HSSPG}, if $X$ is a simplicial set, then $X$ is lower $(2k{-}1)$-Segal if and only if for each gapped set $I\subset S$ of cardinality $k+1$ and each compatible collection $\{f_i\}_{i\in I}$, there exists a unique $F\in X(S)$ with $e_i F = f_i$ for all $i\in I$.
A similar statement holds for the other higher Segal conditions above, and for definitions \ref{def segal properties} and \ref{def d critical} below.
This is the form we most often use these properties later.
\end{remark}

We separate the conditions from \cref{def higher Segal} by simplicial dimension.

\begin{definition}\label{def segal properties}
Let $k,n$ be nonnegative integers and $X$ a simplicial object.
Consider the collection of gapped subsets $I \subset [n]$ of cardinality $k+1$ and the associated collection of cubes $X\cube{I} \colon \ps(I) \to \mathcal{C}$.
We say that $X$ has property
\begin{enumerate}
\item $\lowod^k_n$ if $X\cube{I}$ is cartesian for all $I \subset [n]$,
\item $\lowev^k_n$ if $X\cube{I}$ is cartesian whenever $0 \notin I \subset [n]$,
\item $\uppev^k_n$ if $X\cube{I}$ is cartesian whenever $n \notin I \subset [n]$, and
\item $\uppod^k_n$ if $X\cube{I}$ is cartesian whenever $0,n \notin I \subset [n]$.
\end{enumerate}
\end{definition}

Notice that $X$ satisfies $\lowod^k_n$ or $\uppod^k_n$ if and only if $X^\op$ does so, and $X$ satisfies $\lowev^k_n$ if and only if $X^\op$ satisfies $\uppev^k_n$.
This allows us to infer some theorems by \emph{duality}.
If $k$ is positive,\footnote{Condition $\lowod^0_0$ is vacuous since gapped sets are proper, while $\lowev_1^0$, $\uppev_1^0$, and $\uppod_2^0$ are non-vacuous.} the first non-vacuous conditions are $\lowod^k_{2k}$, $\lowev^k_{2k+1}$, $\uppev^k_{2k+1}$, and $\uppod^k_{2k+2}$.
It will be important to isolate the first \emph{two} simplicial dimensions where there is something to check; see also \cref{rmk d crit}.

\begin{definition}\label{def d critical}
We say that a simplicial set is \emph{upper (resp.\ lower) $d$-critical} if it satisfies the upper (resp.\ lower) $d$-Segal conditions in the lowest two non-vacuous simplicial dimensions ($d+1$ and $d+2$).
This means that our simplicial set is
\begin{enumerate}
\item lower $(2k{-}1)$-critical if it satisfies $\lowod^k_{2k}$ and $\lowod^k_{2k+1}$,
\item lower $2k$-critical if it satisfies $\lowev^k_{2k+1}$ and $\lowev^k_{2k+2}$,
\item upper $2k$-critical if it satisfies $\uppev^k_{2k+1}$ and $\uppev^k_{2k+2}$, and 
\item upper $(2k{+}1)$-critical if it satisfies $\uppod^k_{2k+2}$ and $\uppod^k_{2k+3}$.
\end{enumerate}
\end{definition}

We next have Poguntke's \emph{path space criterion} \cite[Proposition 2.7]{Poguntke:HSSAKT}.\footnote{See also \cite[Theorem 6.3.2]{DyckerhoffKapranov:HSS} and \cite[Theorem 4.10]{GKT1}}
The similar statement replacing `$d$-Segal' with `$d$-critical' holds by \cref{easy psc} below.

\begin{theorem}[Path Space Criterion]\label{path space criterion}
Let $X$ be a simplicial object and $k\geq 1$.
The simplicial object 
$X$ is lower (resp.\ upper) $2k$-Segal if and only if $\decbot X$ (resp.\ $\dectop X$) is lower $(2k{-}1)$-Segal.
The simplicial object $X$ is upper $(2k{+}1)$-Segal if and only if $\decbot X$ is upper $2k$-Segal if and only if $\dectop X$ is lower $2k$-Segal if and only if $\dectop \decbot X = \decbot \dectop X$ is lower $(2k{-}1)$-Segal. \qed
\end{theorem}

\begin{proposition}\label{easy psc}
Let $X$ be a simplicial object and $k\geq 1$.
\begin{enumerate}
		\item $X$ satisfies $\lowev^k_n$ if and only if $\decbot X$ satisfies $\lowod^k_{n-1}$.\label{PSC ldec}
		\item $X$ satisfies $\uppev^k_n$ if and only if $\dectop X$ satisfies $\lowod^k_{n-1}$.\label{PSC udec}
		\item The following are equivalent:
		\begin{enumerate}[label={\alph*.}, ref=\alph*]
		\item $X$ satisfies $\uppod^k_n$.
		\item $\decbot X$ satisfies $\uppev^k_{n-1}$.
		\item $\dectop X$ satisfies $\lowev^k_{n-1}$.
		\item $\decbot \dectop X = \dectop \decbot X$ satisfies $\lowod^k_{n-2}$.
		\end{enumerate}
\end{enumerate}
\end{proposition}
\begin{proof}
As in the proof of Proposition 3.13 of \cite{HackneyLynd:HSSPG}.
\end{proof}

\begin{remark}\label{rmk d crit}
In formulating upper or lower $d$-critical in \cref{def d critical}, one could instead use (for simplicial dimension $n=d+1,d+2$) the geometric conditions in \cite[Definition 3.7]{Dyckerhoff:CPOC} (or \cite[Definition 2.2]{Poguntke:HSSAKT}) expressed in terms of triangulations of cyclic polytopes.
This yields an equivalent definition -- it's not so arduous to prove this directly, but one can also see it by a careful reading of Walde's proof of the equivalence of the $(2k{-}1)$-Segal conditions (especially the proof of Theorem 7.2.1 of \cite{Walde:HSSHE}; see \cite[Remark 3.6]{HackneyLynd:HSSPG}) and Poguntke's account of the path space criterion \cite[Proposition 2.7]{Poguntke:HSSAKT}.
Thus, for example, $2$-critical simplicial objects can be described as those local with respect to triangulations of squares and pentagons.
\end{remark}

We end this section with the Wiggle Lemma, whose strategy of proof informs the argument in \cref{prop lower odd implies coskel}.
Immediate consequences of the Wiggle Lemma are \cite[Lemma 3.9]{HackneyLynd:HSSPG} (on which its proof is based) and \cite[Lemma 3.6]{GKT1}.

\begin{lemma}[Wiggling]\label{wiggling}
If $X$ has property $\lowod^k_{n-1}$ and there is a gapped set $I \subset [n]$ of size $k+1$ such that $X\cube{I}$ is cartesian, then $X$ has property $\lowod^k_n$.
Similarly, if $X$ has property $\lowev_{n-1}^k$ (resp.\ $\uppev_{n-1}^k$, resp.\ $\uppod_{n-1}^k$) and there is a gapped set %(of size $k+1$) 
$I\subset [n]$ avoiding $0$ (resp.\ avoiding $n$, resp.\ avoiding $0$ and $n$) such that $X\cube{I}$ is cartesian, then $X$ has property $\lowev_n^k$ (resp.\ $\uppev_n^k$, resp.\ $\uppod_n^k$).
\end{lemma}
\begin{proof}
If $n < 2k$ then $\lowod^k_n$ is a vacuous condition, while if $n=2k$ then there is exactly one gapped set in $[n]$, so the condition is satisfied by hypothesis. Assume $n > 2k$.

Suppose we have two gapped sets $I, J \subset [n] = S$ which differ in only one position.
In other words, $I\setminus J = \{ \bim \}$ and $J\setminus I = \{ \bjm \}$ such that the successor (resp.\ predecessor) of $\bim$ in $I \cup \{-1, n+1\}$ coincides with the successor (resp.\ predecessor) of $\bjm$ in $J \cup \{-1, n+1\}$.
Notice that $J\subset S_\bim$ and $I \subset S_\bjm$ are also gapped sets.
The following square of maps of cubes commutes.
\[ \begin{tikzcd}[row sep=0.25cm]
& 
X\cube{I_\bim \subset S_\bim} \rar[equals] &[-0.6cm] 
X\cube{J_\bjm \subset S_\bim}
	\drar[bend left=15,"X\cube{J\subset S_\bim}"]
\\
X\cube{I_\bim \subset S} \dar[equals] 
	\urar[bend left=15,"X\cube{I\subset S}"]
& 
&
&
X\cube{J_\bjm \subset S_{\bim,\bjm}} \dar[equals]
\\%[-0.4cm]
X\cube{J_\bjm \subset S}
	\drar[bend right=15,"X\cube{J\subset S}"']
& 
&
&
X\cube{I_\bim \subset S_{\bim,\bjm}}
\\
& 
X\cube{J_\bjm \subset S_\bjm} \rar[equals] &[-0.6cm] 
X\cube{I_\bim \subset S_\bjm}
	\urar[bend right=15,"X\cube{I\subset S_\bjm}"']
\end{tikzcd} \]
The cubes $X\cube{J\subset S_\bim}$ and $X\cube{I\subset S_\bjm}$ are cartesian by $\lowod^k_{n-1}$, so $X\cube{I\subset S}$ is cartesian if and only if $X\cube{J\subset S}$ is cartesian by \cite[Lemma 3.3]{HackneyLynd:HSSPG}.
Iterating this process, we see that $X$ is cartesian for a gapped set $I \subset [n]$ if and only if it is cartesian for $\{0 < 2 < \dots < 2k \}$, and the result follows.
The statements in the second sentence follow from the first via \cref{easy psc}.
\end{proof}

\section{The lower odd Segal case}
\label{sec lower odd}

The next result is strictly about simplicial sets.
See also \cref{hsc implies coskel} below.

\begin{proposition}\label{prop lower odd implies coskel}
If $k\geq 1$ and $X\in \sset$ is lower $(2k{-}1)$-Segal, then $X$ is $2k$-coskeletal.
\end{proposition}
\begin{proof}
Fix $n > 2k$ and let $f_0, \dots, f_n$ be a compatible collection of $(n-1)$-simplices of $X$.
Each gapped subset $I$ of $[n] = S$ gives rise to a unique $n$-simplex $F$ with $e_i F = f_i$ for $i\in I$. 
This immediately implies that there is at most one $n$-simplex having $f_0, \dots, f_n$ as its boundary (since there's at least one gapped set).

The $n$-simplex $F$ does not actually depend on the gapped set $I$; we use a variation of the argument in \cref{wiggling} to show this.
Let $I, J$ be as in that proof, and let $F,F'$ be $n$-simplices such that $e_i F = f_i$ for all $i\in I$, and $e_j F' = f_j$ for all $j\in J$.
As we saw there, $J\subset S_\bim$ is gapped, and we have, for each $j\in J$
\[
	e_j e_\bim F' = e_\bim e_j F' = e_\bim f_j = e_j f_\bim
\]
so we have $e_\bim F' = f_\bim = e_\bim F$.
But now $e_i F' = e_i F$ for all $i\in I$, hence $F'= F$.

Iterating this procedure, we have that if $I' \subset S$ is any gapped set of size $k+1$ and $I = \{0, 2, \dots, 2k\}$, then the $n$-simplex $F'$ associated to $I'$ is equal to the $n$-simplex $F$ associated to $I$.
Since $n$ is strictly larger than $2k$, each $i=0,1,\dots, n$ is a member of \emph{some} gapped set, hence $e_i F = f_i$ for all $i$.
\end{proof}

\begin{remark}\label{rmk extra from proof}
Suppose $n > 2k$ and $X$ satisfies $\lowod_{n-1}^k$ and $\lowod_n^k$.
The preceding proof in fact shows that if $f_0, \dots, f_n$ is a compatible sequence of $(n-1)$-simplices of $X$, then there is a unique $F\in X_n$ with $e_i F = f_i$ for $0 \leq i \leq n$.
\end{remark}

\begin{remark}\label{rmk -1 segal to 0 cosk}
The proof of \cref{prop lower odd implies coskel} fails when $k=0$, and in fact the result is not true.
The `lower $-1$-Segal' sets are the constant simplicial sets, i.e.\ nerves of discrete categories, while $0$-coskeletal simplicial sets are nerves of codiscrete categories.
\end{remark}

\begin{theorem}\label{thm lowod case coskel to segal}
Let $k \geq 1$ and $X\in \sset$.
If $X\in\sset$ is $(2k{+}1)$-coskeletal 
and is lower $(2k{-}1)$-critical, 
then $X$ is lower $(2k{-}1)$-Segal.
\end{theorem}

The proof is by induction: we will show that if a $(2k{+}1)$-coskeletal simplicial set $X$ satisfies $\lowod_{n-1}^k$ and $\lowod_{n-2}^k$ (for $n\geq 2k+2$), then it satisfies $\lowod_n^k$.
Assuming these conditions on $X \in \sset$ for the next few results, we also let $I= \{0, 2, \dots, 2k\} \subset S = [n]$ and we fix a compatible collection $\{f_i \in X(S_i)\}_{i\in I}$.

\begin{lemma}\label{lem end expansion}
Let $M = [2k+2, n] \subset S$.
\begin{enumerate}
\item For each $m\in M$, there exists a unique $f_m \in X(S_m)$ which is compatible with the collection $\{f_i\}_{i\in I}$.
\item The collection $\{f_m\}_{m\in M}$ is compatible.
\end{enumerate}
It follows that $\{f_i\}_{i \in I \cup M}$ is a compatible collection.
\end{lemma}
\begin{proof}
Fix $m\in M$ and notice that $I\subset S_m$ is gapped. Since $\lowod_{n-1}^k$ holds, the compatible collection $\{e_m f_i \in X(S_{i,m}) \}_{i\in I}$ gives a unique $f_m \in X(S_m)$ with $e_i f_m = e_m f_i$ for all $i$.

Suppose $m,m' \in M$. We wish to show that $e_{m'} f_m = e_m f_{m'}$.
But $I\subset S_{m,m'}$ is gapped, so by $\lowod_{n-2}^k$ it suffices to show $e_i e_{m'} f_m = e_i e_m f_{m'}$ for each $i\in I$.
For this, apply \cref{triv lemma} with $t=m$, $p=m'$, and $q=i$.
\end{proof}

\begin{lemma}\label{lem mid expansion}
Let $\{f_i\}_{i \in I \cup M}$ be as in \cref{lem end expansion}, and $J = \{1, 3, \dots, 2k+1\}$.
\begin{enumerate}
\item For each $j\in J$, there exists a unique $f_j \in X(S_j)$ which is compatible with the collection $\{f_i\}_{i \in I \cup M}$. \label{mid expansion one}
\item The collection $\{f_j \}_{j\in J}$ is compatible.\label{mid expansion two}
\end{enumerate}
It follows that $\{f_i\}_{i \in S}$ is a compatible collection. 
\end{lemma}
\begin{proof}
We start with \eqref{mid expansion one}. 
Fix $j\in J$, and let $L \coloneq I_{j-1} \cup \{2k+2\} \subset I \cup M$.
We will show that $L \subset S_{j-1,j}$ is gapped, as is $L \subset S_{j, p}$ for $p > 2k+2$.
For the first statement: 
\begin{itemize}[left=0pt]
\item if $j =1$ then $L = \{2, 4, \dots, 2k+2 \} \subset \{2, 3, 4, \cdots, n \} = S_{0,1}$ is gapped.
\item if $j \geq 3$, then $j-2 \in S_{j-1,j}$ is between $j-3, j+1 \in L$; if $i < i'$ is another adjacent pair in $L$, then $i' \leq j-3$ or $i \geq j+1$, so any $s \in S$ with $i < s < i'$ is still in $S_{j-1,j}$.
\end{itemize}
As $L$ is gapped in $S_{j-1,j}$, it is also gapped in $S_j$; since $p > 2k+2$ is larger than all elements in $L$, this set is also gapped in $S_{j,p}$.

As $\{f_\ell\}_{\ell \in L}$ is compatible (as the restriction of a compatible collection on $I\cup M$), the collection $\{e_j f_\ell \}_{\ell \in L}$ is also compatible.
By $\lowod_{n-1}^k$, there is a unique $f_j \in X(S_j)$ with $e_\ell f_j = e_j f_\ell$ for all $\ell \in L$, i.e.\ there is a unique $f_j$ which is compatible with $\{f_\ell\}_{\ell \in L}$.
If $m\in (I \cup M) \setminus L = \{j-1\} \cup [2k+3,n]$ and $\ell \in L$, then by \cref{triv lemma} (with  $p=m, q=\ell$, and $t=j$) we have $e_\ell e_m f_j = e_\ell e_j f_m$.
We saw in the previous paragraph that $L$ is gapped in $S_{m,j}$, so since $\lowod_{n-2}^k$ holds, we have $e_m f_j = e_j f_m$.
Thus $f_j$ is compatible with $\{f_i \}_{i\in I \cup M}$, establishing \eqref{mid expansion one}.

For \eqref{mid expansion two}, we will show that $e_j f_{j'} = e_{j'} f_j$ for $j < j'$ in $J = \{1, 3, \dots, 2k+1\}$.
Let $L=L^{j,j'}$ be the set of even elements in $[0,j-1] \cup [j'+1,2k+2]$ along with the odd elements in $[j+2,j'-2]$.
The set $L$ has $k+1$ elements and is gapped in $S_{j,j'}$; this is because the inclusion $L\subset S_{j,j'}$ is isomorphic to the inclusion into $[0,n-2]$ of the set of even elements in $[0,2k]$ .
% Generically locally this looks like (j-1), {j+1}, (j+2) ... (j'-2), {j'-1}, (j'+1) 
% or if N = 1 this looks locally like (j-1), {j+1}, (j+3)

To prove $e_j f_{j'} = e_{j'} f_j$, we induct on $N = (j'-j)/2 \geq 1$.
By $\lowod_{n-2}^k$ it suffices to show $e_\ell e_j f_{j'} = e_\ell e_{j'} f_j$ for all $\ell \in L^{j,j'}$.
Notice that the even elements of $L^{j,j'}$ are contained in $I\cup M$, so this equation holds in that case by \cref{triv lemma} (with $q=\ell$, $p=j$, and $t=j'$).
This completely takes care of the $N=1$ case, as $L=L^{j,j+2}$ is the set of even elements in $[0,j-1] \cup [j+3, 2k+2]$.
For  $N > 1$, we must consider the case of an odd element $\ell$ in $[j+2,j'-2]$.
The inductive step gives $e_\ell f_j = e_j f_\ell$ and $e_\ell f_{j'} = e_{j'} f_\ell$. 
By \cref{triv lemma} (with $q=\ell$, $t=j'$, and $p=j$), we have $e_\ell e_j f_{j'} = e_\ell e_{j'} f_j$. 
We conclude that $e_j f_{j'} = e_{j'} f_j$ for $j,j' \in J$.
\end{proof}

\begin{proposition}\label{prop lowod unique fill}
Under the standing assumptions above, there exists a unique $F\in X_n$ such that $e_i F = f_i$ for all $i\in I$.
\end{proposition}
\begin{proof}
Existence is established by the previous lemmas. 
But so is uniqueness: suppose $F, F' \in X_n$ are such that $e_i F = e_i F' = f_i$ for all $i\in I$.
By uniqueness of $f_m$ and of $f_j$ in the previous two lemmas, we have $e_m F = e_m F' = f_m$ for $m\in M$ and $e_j F = e_j F' = f_j$ for $j\in J$.
Since $X$ is $(2k{+}1)$-coskeletal, $F = F'$. So such an $F$, if it exists, is unique.
\end{proof}

\begin{proof}[Proof of \cref{thm lowod case coskel to segal}]
Suppose $X$ is $(2k{+}1)$-coskeletal and satisfies $\lowod_{n-1}^k$ and $\lowod_{n-2}^k$ (for $n\geq 2k+2$).
\Cref{prop lowod unique fill} and \cref{wiggling} imply that $X$ satisfies $\lowod_n^k$. By induction, $X$ is lower $(2k{-}1)$-Segal.
\end{proof}

\section{Coskeletality and d\'ecalage}
We now pursue analogues of \cref{prop lower odd implies coskel} and \cref{thm lowod case coskel to segal} for the other higher Segal conditions (\cref{hsc implies coskel} and \cref{thm other cases coskel to segal}), leading to \cref{thm main theorem}.
It is possible to simply adjust the proofs in \cref{sec lower odd} to cover the additional three cases.
Instead, we combine the path space criterion with a close relationship between the coskeletality of $X$ and its d\'ecalages (\cref{prop coskel to dec coskel} and \cref{lem dec coskel to coskel}).
Before turning to this, we take care of an exceptional case by hand.

\begin{lemma}\label{lem k=0 case}
If $X$ is upper or lower $0$-critical and $2$-coskeletal, or if $X$ is upper $1$-critical and $3$-coskeletal, then $X$ is constant.
% Consequently, $X$ is $0$-Segal and $\dectop X \cong \decbot X$ is $1$-coskeletal.
\end{lemma}
\begin{proof}
We assume $X$ is upper 1-critical and 3-coskeletal, with the other cases being similar.
Upper $1$-critical means $d_1 \colon X_2 \to X_1$ and $d_1,d_2 \colon X_3 \to X_2$ are bijections.
From the simplicial identities one then checks that all parallel face/degeneracy maps $X_n \to X_{n \pm 1}$ for $n, n\pm 1 \in \{1,2,3\}$ are equal bijections.
Then $\operatorname{id}_{X_1} = d_1s_0 = d_2 s_0 = s_0 d_1$.
This is enough to infer that $d_0 = d_1 \colon X_1 \cong X_0$ with inverse $s_0$.
Thus $X|_{\Delta_{\leq 3}}$ is isomorphic to the constant presheaf on $X_0$, so if $X$ is $3$-coskeletal it is isomorphic to the constant simplicial set on $X_0$.
\end{proof}

\begin{proposition}\label{prop coskel to dec coskel}
Let $n\geq 0$.
If $X$ is $n$-coskeletal, then so are its d\'ecalages $\dectop X$ and $\decbot X$.
Further, still assuming $X$ is $n$-coskeletal ($k$ a nonnegative integer):
\begin{enumerate}
\item If $n=2k+2$ and $X$ is lower $2k$-critical, 
then $\decbot X$ is $(n{-}1)$-coskeletal.
\item If $n=2k+2$ and $X$ is upper $2k$-critical, 
then $\dectop X$ is $(n{-}1)$-coskeletal.\label{item upev}
\item If $n=2k+3$ and $X$ is upper $(2k{+}1)$-critical,
then $\dectop X$ and $\decbot X$ are $(n{-}1)$-coskeletal. \label{item upod}
\end{enumerate}
\end{proposition}
\begin{proof}
We will only consider $\dectop X$ in this proof; the results about $\decbot$ follow by duality.
Fix $p > n-1$ and suppose $f_0, \dots, f_p$ is a compatible collection of $(p-1)$-simplices of $\dectop X$.
More specifically, set $S = [p+1]$, $S' = S_{p+1} = [p]$, then we're taking $f_i \in (\dectop X)(S'_i) = X(S_i) \cong X_{p}$ such that $e_i f_j = e_j f_i$ for $0 \leq i < j \leq p$.
We must show there exists a unique $F \in (\dectop X)(S') = X(S)$ such that $e_i F = f_i$ for $0\leq i \leq p$.

There is a compatible collection $e_{p+1} f_0, \dots, e_{p+1} f_p$ of $(p-1)$-simplices of $X$. % proof omitted
If $p > n$, then since $X$ is $n$-coskeletal, there exists a unique $f_{p+1} \in X(S') = X(S_{p+1})$ such that $e_i f_{p+1} = e_{p+1} f_i$ for all $0 \leq i \leq p$.
The collection $f_0, \dots, f_p, f_{p+1}$ is a compatible collection of $p$-simplices of $X$, so there exists a unique $F \in X(S)$ such that $e_i F = f_i$ for $i=0, \dots, p+1$.
If there is another $F' \in X(S)$ such that $e_i F' = f_i$ for $i=0, \dots, p$, then also $e_{p+1} F' = f_{p+1}$ by uniqueness of $f_{p+1}$.
Hence $F=F'$ since $X$ is $n$-coskeletal.
This establishes $n$-coskeletality for $\dectop X$.

It remains to address the $p=n$ case when a hypothesis of \eqref{item upev} or \eqref{item upod} holds.
When $k=0$, we have $\dectop X$ is constant by \cref{lem k=0 case} (since $X$ is so), hence $1$-coskeletal. 
We may thus assume $k \geq 1$.
We begin with \eqref{item upev} and assume $p=n=2k+2$ and $X$ satisfies $\uppev^k_{p-1}$ and $\uppev^k_p$. 
Then $\dectop X$ satisfies $\lowod^k_{p-2} = \lowod^k_{2k}$ and $\lowod^k_{p-1} = \lowod^k_{2k+1}$.
By \cref{rmk extra from proof} applied to $\dectop X$, there is a unique $f_{p+1}$ such that $e_i f_{p+1} = e_{p+1} f_i$ for $0 \leq i \leq p-1$.
Why is it the case that $e_p f_{p+1} = e_{p+1} f_p \in X(S_{p,p+1}) \cong X_{p-1}$?
We let $I = \{0, 2, 4, \dots, p-2=2k\} \subset S_{p,p+1}$ be the unique gapped set missing the maximal element $p-1=2k+1$.
For $i\in I$ we have % $e_i f_p = e_p f_i$ and $e_i f_{p+1} = e_{p+1} f_i$, hence 
$e_i e_p f_{p+1} = e_i e_{p+1} f_p$ by \cref{triv lemma}, so $e_p f_{p+1} = e_{p+1} f_p$ by $\uppev_{p-1}^k$. 
Thus $f_0, \dots, f_{p+1}$ is a compatible collection of $p$-simplices of $X$. 
Since $p+1 = n+1 > n$ there exists a unique $F\in X(S)$ such that $e_i F = f_i$ for $i \leq p+1$.
Uniqueness of $F$ satisfying $e_i F = f_i$ for $i \leq p$ is exactly as in the previous paragraph.
Hence $\dectop X$ is $(n{-}1)$-coskeletal.

Finally we come to \eqref{item upod}, which is similar to the preceding paragraph.
Assume $p=n = 2k+3$ and $X$ satisfies $\uppod_{p-1}^k$ and $\uppod_p^k$.
Then $Y = \decbot \dectop X = \dectop\decbot X$ satisfies $\lowod_{p-3}^k$ and $\lowod_{p-2}^k$.
By \cref{rmk extra from proof} applied to $Y$, there is a unique $f_{p+1}$ such that $e_i f_{p+1} = e_{p+1} f_i$ for $1\leq i \leq p-1$.
We wish to show $e_0 f_{p+1} = e_{p+1} f_0$ and $e_p f_{p+1} = e_{p+1} f_p$.
In the first case, let $I = \{2, 4, \dots, 2k+2\} \subset S_{0,p+1} = S_{0,2k+4}$ be the unique gapped set missing the maximal and minimal elements; 
% $e_i f_{p+1} = e_{p+1} e_i$ and $e_i f_0 = e_0 f_i$
by \cref{triv lemma} we have $e_i e_0 f_{p+1} = e_i e_{p+1} f_0$ for each $i\in I$, hence $e_0 f_{p+1} = e_{p+1} f_0$.
Likewise, in the second case, we use the gapped set $I = \{1,3, \dots, 2k+1\} \subset S_{p,p+1} = S_{2k+3,2k+4}$ to see $e_p f_{p+1} = e_{p+1} f_p$.
Thus $f_0, \dots, f_{p+1}$ is a compatible collection of $X$, hence there is a unique $(p+1)$-simplex $F$ of $X$ such that $e_i F = f_i$ for $0\leq i \leq p+1$.
This establishes existence, uniqueness is similar to previous arguments.
We conclude that $\dectop X$ is $(n{-}1)$-coskeletal under the hypotheses of \eqref{item upod}.
\end{proof}

\begin{lemma}\label{lem dec coskel to coskel}
Let $X$ be a simplicial set and $n \geq 0$. 
If either $\dectop X$ or $\decbot X$ is $n$-coskeletal, then $X$ is $(n{+}2)$-coskeletal.
If, additionally, either 
\begin{enumerate}
\item $n=2k$ and $X$ satisfies $\uppev_{2k+1}^k$ or $\lowev_{2k+1}^k$, or
\item $n=2k+1$ and $X$ satisfies $\uppod_{2k+2}^k$,
\end{enumerate}
then $X$ is $(n{+}1)$-coskeletal.
\end{lemma}
In particular, the first condition holds if $X$ is upper or lower $2k$-critical, while the second holds if $X$ is upper $(2k{+}1)$-critical.

\begin{proof}
By duality it suffices to prove the result when $\dectop X$ is $n$-coskeletal. 
Let $p > n+1$, and suppose $f_0, \dots, f_p$ are compatible $(p-1)$-simplices of $X$.
Our aim is to show that there is a unique element $F\in X_p$ such that $e_i F = f_i$ for $i=0,\dots, p$.
If $F,F' \in X_p = (\dectop X)_{p-1}$ both have this property, then $e_i F = e_i F'$ for $i \leq p-1$; since $p-1 > n$, coskeletality of $\dectop X$ implies $F = F'$. This establishes uniqueness.

We turn to existence.
The elements $f_0, \dots, f_{p-1}$ are compatible $(p{-}2)$-simplices of $\dectop X$, so there exists a unique $F \in (\dectop X)_{p-1} = X_p$ such that $e_i F = f_i$ for $i=0, \dots, p-1$.
We must show $e_p F = f_p$.
We have equality of compatible collections
\[
	\{ e_i f_p \}_{0 \leq i \leq p-1} = \{ e_p f_i \}_{0 \leq i \leq p-1} = \{ e_p e_i F \}_{0 \leq i \leq p-1} = \{ e_i (e_p F) \}_{0 \leq i \leq p-1}
\]
of $(p-2)$-simplices of $X$.
This implies $\{ e_i f_p \}_{0 \leq i \leq p-2} = \{ e_i (e_p F) \}_{0 \leq i \leq p-2}$ is a compatible collection of $(p-3)$-simplices of $\dectop X$, so if $p-2 > n$ we have $f_p = e_p F$ by $n$-coskeletality of $\dectop X$.
In particular, we have established that $X$ is $(n{+}2)$-coskeletal.

We are now left with the situation $p=n+2$, assuming one of the additional conditions from the statement, namely that $X$ satisfies one of the conditions $\lowev^k_{p-1}$, $\uppev^k_{p-1}$, or $\uppod^k_{p-1}$, where $k= \lfloor n/2 \rfloor$.
In each case there is exactly one $(k+1)$-element gapped subset $I\subset [p-1]$ which misses $0$,  $p-1$, or both.
By the uniqueness part of the additional assumption, since $\{ e_i f_p \}_{i\in I} = \{ e_i (e_p F) \}_{i\in I}$, we must have $f_p = e_p F$.
Thus $X$ is $(n{+}1)$-coskeletal.
\end{proof}

\begin{theorem}\label{hsc implies coskel}
If a simplicial set is upper or lower $d$-Segal for some $d\geq 0$, then it is $(d{+}1)$-coskeletal.
\end{theorem}
\begin{proof}
Let $X\in \sset$.
If $X$ is upper or lower $0$-Segal, or if $X$ is upper $1$-Segal, then it is constant, hence $1$-coskeletal.
Now suppose $k \geq 1$.
If $X$ is upper $2k$-Segal, then $\dectop X$ is lower $(2k{-}1)$-Segal.
By \cref{prop lower odd implies coskel}, $\dectop X$ is $2k$-coskeletal.
It follows from \cref{lem dec coskel to coskel} that $X$ is $(2k{+}1)$-coskeletal.
If $X$ is lower $2k$-Segal, then using duality we infer that $X$ is $(2k{+}1)$-coskeletal.
If $X$ is upper $(2k{+}1)$-Segal, then $\dectop X$ is lower $2k$-Segal, hence $(2k{+}1)$-coskeletal.
By \cref{lem dec coskel to coskel}, $X$ is $(2k{+}2)$-coskeletal.
The remaining case is \cref{prop lower odd implies coskel}.
\end{proof}

In some cases one may prove that certain $d$-Segal sets enjoy a much lower coskeletality than $d+1$. 
For instance, in the case of `partial groups', $2k$-Segal implies $k$-coskeletal \cite[Theorem 6.6]{Hackney:PGSO}.
We also have the following:

\begin{remark}
Suppose $X$ is a simplicial abelian group, and let $C$ be the associated normalized chain complex.
Theorem 4.12 of  \cite{DyckerhoffJassoWalde:SSHART} implies that $X$ is $2k$-Segal if and only if $C_m = 0$ for $m > k$. 
Meanwhile, $X$ is $n$-coskeletal (as a simplicial set) if and only if $C_m = 0$ for $m > n+1$ and $H_mC = 0$ for $m\geq n$ \cite[\href{https://kerodon.net/tag/0528}{Tag 0528}]{kerodon}.
Thus each $2k$-Segal simplicial abelian group is $(k{+}1)$-coskeletal.
\end{remark}

\begin{theorem}\label{thm other cases coskel to segal}
Let $k \geq 0$ and $X\in \sset$.
\begin{enumerate}
\item If $X$ is $(2k{+}2)$-coskeletal and lower $2k$-critical, 
then $X$ is lower $2k$-Segal.\label{item other case lowev}
\item If $X$ is $(2k{+}2)$-coskeletal and upper $2k$-critical, 
then $X$ is upper $2k$-Segal.\label{item other case uppev}
\item If $X$ is $(2k{+}3)$-coskeletal and upper $(2k{+}1)$-critical, 
then $X$ is upper $(2k{+}1)$-Segal.\label{item other case uppod}
\end{enumerate}
\end{theorem}
\begin{proof}
If $k = 0$, then $X$ is constant by \cref{lem k=0 case}, hence is $0$-Segal. Suppose $k \geq 1$.
Assuming that $X$ is $(2k{+}2)$-coskeletal and 
upper $2k$-critical, 
we have that $\dectop X$ is $(2k{+}1)$-coskeletal and 
lower $(2k{-}1)$-critical
by \cref{prop coskel to dec coskel} and \cref{easy psc}.
It follows that $\dectop X$ is lower $(2k{-}1)$-Segal by \cref{thm lowod case coskel to segal}. Thus $X$ is upper $2k$-Segal by the path space criterion, so we have established \eqref{item other case uppev}.
Item \eqref{item other case lowev} follows from \eqref{item other case uppev} by duality, while \eqref{item other case uppod} is proved in an analogous way (using \eqref{item other case lowev} in place of \cref{thm lowod case coskel to segal}).
\end{proof}

\begin{theorem}\label{thm main theorem}
Let $X$ be a simplicial set and $d \geq 0$.
If $X$ is $(d{+}2)$-coskeletal and is upper (resp.\ lower) $d$-critical, then it is upper (resp.\ lower) $d$-Segal.
Conversely, if $X$ is upper or lower $d$-Segal, then it is $(d{+}1)$-coskeletal. \qed
\end{theorem}

\begin{remark}\label{rmk simplicial anima}
The corresponding statement for simplicial spaces is false.
Let $M$ be a nonempty space, and $X$ the associated constant simplicial space, which is automatically $0$-Segal (hence $d$-Segal for all $d$).
The $p$-coskeleton of $X$ (see \cite[\href{https://kerodon.net/tag/05GU}{Tag 05GU}]{kerodon}) is given in dimension $n$ by 
\[
	(\cosk_p X)_n \simeq \lim_{(\Delta_{/[n]} \times_\Delta \Delta_{\leq p})^\op} M \simeq \map(|\Delta_{/[n]} \times_\Delta \Delta_{\leq p}|, M),
\]
the second equivalence because $X$ is constant, 
where $\Delta_{\leq p} \subset \Delta$ is the full subcategory on $[q]$ for $q \leq p$.
Since $|\Delta_{/[n]} \times_\Delta \Delta_{\leq p}| \simeq |\sk_p \Delta^n|$, taking $n=p+1$ gives $|\partial \Delta^{p+1}| \simeq S^p$, and the unit map $X_{p+1} \to (\cosk_p X)_{p+1}$ is identified with restriction along $\sk_p \Delta^{p+1} \subset \Delta^{p+1}$, i.e.\ with the inclusion of constant maps $M \to M^{S^p}$.
This is a section of evaluation at a basepoint, whose fibers are $\Omega^p M$, hence is an equivalence if and only if $\Omega^p M$ is contractible.

Now suppose $M$ has infinitely many nontrivial homotopy groups (e.g.\ $M=S^2$).
Then $X$ immediately gives a counterexample to the second statement, as $\Omega^p M$ is not contractible for any $p$, but $X$ is $d$-Segal for all $d\geq 0$.

For the first statement, set $Y \coloneq \cosk_{d+2} X$, which is $(d{+}2)$-coskeletal. 
As the unit $X \to Y$ is an equivalence in simplicial dimensions up to $d+2$, we have that $Y$ is $d$-critical since $X$ is so.
But naturality implies that $Y$ cannot satisfy the $d$-Segal condition in simplicial dimension $d+3$, since $X$ is known to, and $M \simeq X_{d+3} \to Y_{d+3} \simeq M^{S^{d+2}}$ is not an equivalence.
\end{remark}

\bibliographystyle{amsplain}
\bibliography{higher}
\end{document}